\documentclass[12pt]{article}

\def\UseBibLatex{1}

\makeatletter
\def\input@path{{styles/}}
\makeatother

\newcommand{\UsePackage}[1]{%
  \IfFileExists{styles/#1.sty}{%
      \usepackage{styles/#1}%
   }{%
      \IfFileExists{../styles/#1.sty}{%
         \usepackage{../styles/#1}%
      }{%
         \usepackage{#1}%
      }%
   }%
}

\usepackage[T1]{fontenc}
\usepackage{lmodern}
\usepackage{textcomp}
\usepackage{bbm}
\usepackage{amsmath}%
\usepackage{amssymb}%
\usepackage[table]{xcolor}%

\usepackage{todonotes}
\usepackage[in]{fullpage}%

\usepackage[amsmath,thmmarks]{ntheorem}%
\theoremseparator{.}%

\usepackage{titlesec}%
\titlelabel{\thetitle. }%
\usepackage{xcolor}%
\usepackage{mleftright}%
\usepackage{xspace}%
\usepackage{graphicx}
\usepackage{hyperref}%

\usepackage{hyperref}%
\hypersetup{%
      unicode,
      breaklinks,%
      colorlinks=true,%
      urlcolor=[rgb]{0.25,0.0,0.0},%
      linkcolor=[rgb]{0.5,0.0,0.0},%
      citecolor=[rgb]{0,0.2,0.445},%
      filecolor=[rgb]{0,0,0.4},
      anchorcolor=[rgb]={0.0,0.1,0.2}%
}
\usepackage[ocgcolorlinks]{ocgx2}

\theoremseparator{.}%

\theoremstyle{plain}%
\newtheorem{theorem}{Theorem}[section]

\newtheorem{lemma}[theorem]{Lemma}

\newtheorem{corollary}[theorem]{Corollary}

\newtheorem{question}[theorem]{Question}
\newtheorem{proposition}[theorem]{Proposition}

\theoremstyle{plain}%
\theoremheaderfont{\sf} \theorembodyfont{\upshape}%
\newtheorem*{remark:unnumbered}[theorem]{Remark}%
\newtheorem{remark}[theorem]{Remark}%
\newtheorem{definition}[theorem]{Definition}

\newtheorem{example}[theorem]{Example}

\newcommand{\myqedsymbol}{\rule{2mm}{2mm}}

\theoremheaderfont{\em}%
\theorembodyfont{\upshape}%
\theoremstyle{nonumberplain}%
\theoremseparator{}%
\theoremsymbol{\myqedsymbol}%
\newtheorem{proof}{Proof:}%

\newtheorem{proofof}{Proof of\!}%

\providecommand{\emphind}[1]{}%
\renewcommand{\emphind}[1]{\emph{#1}\index{#1}}

\definecolor{blue25emph}{rgb}{0, 0, 11}

\providecommand{\emphic}[2]{}
\renewcommand{\emphic}[2]{\textcolor{blue25emph}{%
      \textbf{\emph{#1}}}\index{#2}}

\providecommand{\emphi}[1]{}%
\renewcommand{\emphi}[1]{\emphic{#1}{#1}}

\definecolor{almostblack}{rgb}{0, 0, 0.3}

\providecommand{\emphw}[1]{}%
\renewcommand{\emphw}[1]{{\textcolor{almostblack}{\emph{#1}}}}%

\providecommand{\emphOnly}[1]{}%
\renewcommand{\emphOnly}[1]{\emph{\textcolor{blue25}{\textbf{#1}}}}

\newcommand{\RaghavThanks}[1]{%
   \thanks{%
      Department of Mathematics; %
      University of Washington; %
      Seattle, USA;
      \href{mailto:raghavt@uw.edu}{raghavt@uw.edu}. %
   #1%
   }%
}

\newcommand{\HLink}[2]{\hyperref[#2]{#1~\ref*{#2}}}
\newcommand{\HLinkSuffix}[3]{\hyperref[#2]{#1\ref*{#2}{#3}}}

\newcommand{\remlab}[1]{\label{rem:#1}}
\newcommand{\remref}[1]{\HLink{Remark}{rem:#1}}%

\providecommand{\deflab}[1]{}
\renewcommand{\deflab}[1]{\label{def:#1}}

\newcommand{\lemlab}[1]{\label{lemma:#1}}
\newcommand{\lemref}[1]{\HLink{Lemma}{lemma:#1}}%

\providecommand{\eqlab}[1]{}%
\renewcommand{\eqlab}[1]{\label{equation:#1}}
\newcommand{\Eqref}[1]{\HLinkSuffix{Eq.~(}{equation:#1}{)}}

\newcommand{\remove}[1]{}%

\usepackage[inline]{enumitem}

\newlist{compactenumA}{enumerate}{5}%
\setlist[compactenumA]{topsep=0pt,itemsep=-1ex,partopsep=1ex,parsep=1ex,%
   label=(\Alph*)}%

\newlist{compactenuma}{enumerate}{5}%
\setlist[compactenuma]{topsep=0pt,itemsep=-1ex,partopsep=1ex,parsep=1ex,%
   label=(\alph*)}%

\newlist{compactenumI}{enumerate}{5}%
\setlist[compactenumI]{topsep=0pt,itemsep=-1ex,partopsep=1ex,parsep=1ex,%
   label=(\Roman*)}%

\newlist{compactenumi}{enumerate}{5}%
\setlist[compactenumi]{topsep=0pt,itemsep=-1ex,partopsep=1ex,parsep=1ex,%
   label=(\roman*)}%

\newlist{compactitem}{itemize}{5}%
\setlist[compactitem]{topsep=0pt,itemsep=-1ex,partopsep=1ex,parsep=1ex,%
   label=\ensuremath{\bullet}}%

\providecommand{\BibLatexMode}[1]{}
\providecommand{\BibTexMode}[1]{}

\ifx\UseBibLatex\undefined%
  \renewcommand{\BibLatexMode}[1]{}
  \renewcommand{\BibTexMode}[1]{#1}
\else
  \renewcommand{\BibLatexMode}[1]{#1}
  \renewcommand{\BibTexMode}[1]{}
\fi

\BibLatexMode{%
   \usepackage[bibencoding=utf8,style=alphabetic,backend=biber]{biblatex}%
   \UsePackage{my_biblatex}%
}

\numberwithin{figure}{section}%
\numberwithin{table}{section}%
\numberwithin{equation}{section}%

\newcommand{\R}{\mathbb{R}}
\newcommand{\Bcal}{\mathcal{B}}
\newcommand{\Ecal}{\mathcal{E}}
\newcommand{\Hcal}{\mathcal{H}}
\newcommand{\maxtr}[1]{\mathrm{maxTrace}\left(#1\right)}
\newcommand{\Pcal}{\mathcal{P}}
\newcommand{\N}{\mathbb{N}}

\BibLatexMode{\bibliography{reference}}
\usepackage{mathtools,amsfonts,amsmath,amssymb}

\usepackage[math]{cellspace}
\providecommand{\keywords}[1]{\textbf{\textit{Keywords---}} #1}
\providecommand{\subjclass}[1]{{\textit{2020 Mathematics Subject Classification : }} #1\newline}
\begin{document}

\title{Some observations on Erd\H{o}s matrices}

\author{
   Raghavendra Tripathi
   \RaghavThanks{}
}

\maketitle

\begin{abstract}
   In a seminal paper in 1959, Marcus and Ree proved that every $n\times n$ bistochastic matrix $A$ satisfies $\|A\|_{\operatorname{F}}^2\leq \max_{\sigma\in S_n}A_{i,\sigma(i)}$ where $S_n$ is the symmetric group on $\{1, \ldots, n\}$. Erd\H{o}s asked to characterize the bistochastic matrices for which the equality holds in the Marcus--Ree inequality. We refer to such matrices as Erd\H{o}s matrices. While this problem is trivial in dimension $n=2$, the case of dimension $n=3$ was only resolved recently in~\cite{bouthat2024question} in 2023. We prove that for every $n$, there are only finitely many $n\times n$ Erd\H{o}s matrices. We also give a characterization of Erd\H{o}s matrices that yields an algorithm to generate all Erd\H{o}s matrices in any given dimension. We also prove that Erd\H{o}s matrices can have only rational entries. This answers a question of~\cite{bouthat2024question}.  
\end{abstract}

\subjclass{15A15, 15A45, 15B36, 15B51}
\keywords{Bistochastic matrix, Doubly stochastic matrix, maximal trace, Frobenius norm}

\section{Introduction}
\label{sec:Intro}
A bistochastic matrix $A\in \R^{n\times n}$ is an $n\times n$ matrix with non-negative entries such that the entries in each row and each column sum to $1$, that is, $\sum_{k=1}^{n} A_{i, k}=1=\sum_{k=1}^{n}A_{k, i}$ for all $i\in [n]\coloneqq \{1, \ldots, n\}$. The famous Birkhoff--von Neumann theorem~\cite{birkhoff1946three} states that the set of all $n\times n$ bistochastic matrices, $\Bcal_n$, is the convex hull generated by the set $\Pcal_n$ of $n\times n$ permutation matrices. Bistochastic matrices are ubiquitous objects in probability theory, graph theory, and many more areas of mathematics. Naturally, they have been extensively studied and continue to be investigated even today.

In 1959, Marcus and Ree~\cite{marcus1959diagonals} proved that any bistochastic matrix $A\in \R^{n\times n}$ satisfies
\begin{equation}
\eqlab{MarcusRee}
\|A\|_{\operatorname{F}}^2\leq \max_{\sigma\in S_n} \sum_{i=1}^{n}A_{i, \sigma(i)}\;,
\end{equation}
where $S_n$ is the set of all permutations of the set $[n]$ and $\|\cdot\|_{\operatorname{F}}$ denotes the usual Frobenius norm of a matrix, that is, $\|A\|_{\operatorname{F}}^2 = \sum_{i, j=1}^{n}A_{i, j}^2$.  It is well-known that the Frobenius norm is induced by the Frobenius inner-product $\langle A, B\rangle_{\operatorname{F}}\coloneqq \mathrm{Tr}(AB^{T})$ on the space of $n\times n$ real matrices. Here $B^{T}$ denotes the transpose of the matrix $B$. The proof of~\Eqref{MarcusRee} is an immediate consequence of Birkhoff--von Neumann theorem and the fact that $B\mapsto \langle A, B\rangle_{\operatorname{F}}$ is a convex function. 

Following~\cite{bouthat2024question}, we refer to the quantity $\max_{\sigma\in S_n}A_{i,\sigma(i)}$ as the maximal trace of $A$ and denote it by $\maxtr{A}$. For later use, we record 
\[\maxtr{A}\coloneqq \max_{\sigma\in S_n} \sum_{i=1}^{n}A_{i, \sigma(i)}=\max_{P\in \Pcal_n} \langle A, P\rangle_{\operatorname{F}}.\] 
This quantity has been investigated in the context of assignment problems. We refer the reader to~\cite{balasubramanian1979maximal,wang1974maximum,brualdi2022diagonal} and the references therein for more detail. 
Seeing the inequality~\Eqref{MarcusRee}, Erd\H{o}s asked the following question.
\begin{question}
\label{ques:Erdos}
    Characterize the matrices $A\in \Bcal_n$ such that $\|A\|_{\operatorname{F}}^2 = \maxtr{A}.$
\end{question}
Throughout this article, we will refer to such matrices as $0$-Erd\H{o}s matrices (for the reason that will be apparent soon) or simply as the Erd\H{o}s matrices.  Since both the functions $A\mapsto \|A\|_{\operatorname{F}}^2$ and $A\mapsto \maxtr{A}$ remain unchanged if $A$ is replaced by $PAQ$ for some permutation matrices $P$ and $Q$, it is clear that if $A$ is an Erd\H{o}s matrix then so are $PAQ$ for any permutation matrices $P$ and $Q$. Throughout this paper, we will say $A$ and $B$ are \emph{equivalent}, denoted $A\sim B$, if  $A=PBQ$ for some permutation matrices $P$ and $Q$. It makes sense to characterize the Erd\H{o}s matrices up to the equivalence relation $A\sim B$.

Characterizing $2\times 2$ Erd\H{o}s matrices is trivial.  It is easily verified (see~Section \ref{par:twobytwo}) that (up to the equivalence) there are precisely two Erd\H{o}s matrices in $\Bcal_2$, namely, 
\[
I_2 = \begin{pmatrix}
1 & 0 \\ 0 & 1
\end{pmatrix}\;, \quad\quad J_2 = \begin{pmatrix} \frac{1}{2} & \frac{1}{2}\\ \frac{1}{2} & \frac{1}{2} \end{pmatrix}.
\]
These two examples naturally generalize to the arbitrary dimensions. It is easy to check that the identity matrix $I_n\in \Bcal_n$ and the $n\times n$ matrix $J_n\in \Bcal_n$ all of whose entries are $\frac{1}{n}$ are Erd\H{o}s matrices. Marcus and Ree~\cite{marcus1959diagonals} gave several other examples of Erd\H{o}s matrices for general dimension $n$ and proved some interesting partial results.  However, a satisfactory resolution of the problem remained still elusive. While the bistochastic matrices continued to be investigated in various contexts, this particular problem (Question~\ref{ques:Erdos}) seems to have been forgotten until recently. A complete characterization of $3\times 3$ Erd\H{o}s matrices has only been obtained recently in~\cite{bouthat2024question}.  It is shown in~\cite{bouthat2024question} that, up to equivalence, there are precisely $6$ Erd\H{o}s matrices in $\Bcal_3$. For completeness and the reader's convenience, we describe these $6$ matrices below: 
\begin{align*}
I_n &= \begin{pmatrix}
1 & 0 & 0 \\ 0 & 1& 0\\ 0 & 0 & 1
\end{pmatrix}, \quad J_3 = \frac{1}{3}\begin{pmatrix}
1 & 1 & 1\\ 1 & 1 & 1\\ 1& 1 &1
\end{pmatrix}, \quad I\oplus J_2 = \begin{pmatrix}
1 & 0 & 0 \\ 0 & \frac{1}{2} & \frac{1}{2}\\ 0 & \frac{1}{2} & \frac{1}{2}
\end{pmatrix},\\
S &= \begin{pmatrix}
   0 & \frac{1}{2} & \frac{1}{2}\\
   \frac{1}{2} & \frac{1}{4} & \frac{1}{4}\\
   \frac{1}{2} & \frac{1}{4} & \frac{1}{4}
\end{pmatrix}, \quad T= \begin{pmatrix}
    0 & \frac{1}{2} & \frac{1}{2}\\
    \frac{1}{2} & 0 & \frac{1}{2}\\
    \frac{1}{2} & \frac{1}{2} & 0
\end{pmatrix}, \quad R=\begin{pmatrix}
    \frac{3}{5} & 0 & \frac{2}{5}\\
    0 & \frac{3}{5} & \frac{2}{5}\\
    \frac{2}{5} & \frac{2}{5} & \frac{1}{5}
\end{pmatrix}\;.
\end{align*}
The matrix $I\oplus J_2$ and $T=\frac{1}{2}(3J_3-I_3)$ also generalize to higher dimensions, that is, $\frac{1}{n-1}(nJ_n-I_n)\in \Bcal_n$ is an Erd\H{o}s matrix. The matrices $S$ and $R$ are essentially the new contributions from~\cite{bouthat2024question} and the authors remark that it is unclear if the matrices $S$ and $R$ generalize to the higher dimensions in any natural way.

While all Erd\H{o}s matrices in dimensions $n=2, 3$ are equivalent to some symmetric matrix, this is not true in general. In~\cite{bouthat2024question}, the authors produce the following example of a $4\times 4$ Erd\H{o}s matrix that is not equivalent to any symmetric matrix:
\[\frac{1}{6}\begin{pmatrix}
    3 & 3 & 0 & 0 \\
    1 & 1 & 2 & 2 \\
    1 & 1 & 2 & 2 \\
    1 & 1 & 2 & 2
\end{pmatrix}\;.\]
We highly recommend~\cite{bouthat2024question} for a clear overview and the history of this problem. The discussion so far naturally suggests the following problems, that we answer in this paper.
\begin{enumerate}
  \item Are there only finitely many Erd\H{o}s matrices in $\Bcal_n$ for each $n$?
  \item (~\cites[Question 3]{bouthat2024question}) Do Erd\H{o}s matrices have only rational entries? 
\end{enumerate}
We answer the above question affirmatively in Theorem~\ref{thm:MainTheorem}. Theorem~\ref{thm:MainTheorem} builds on the following proposition that gives a characterization of the Erd\H{o}s matrices that is of independent interest. 
\begin{proposition}
\label{prop:Characterization}
Let $n\in \N$ and let $A=\sum_{i=1}^{m}x_iP_i\in \mathcal{B}_n$ where $P_i$ are $n\times n$ permutation matrices and $x_i>0$ for all $i\in [m]$ such that $\sum_{i=1}^{m}x_i=1$. If $A$ is an Erd\H{o}s matrix, then ${\bf x}=(x_1, \ldots, x_m)^{T}
\in \R^m$ solves the system of equations $M{\bf x}=\langle M{\bf x}, {\bf x}\rangle \mathbbm{1}_m$ where $M\in \mathbb{R}^{m\times m}$ is a matrix such that $M_{i, j}= \langle P_i, P_j\rangle_{\operatorname{F}}$ and $\mathbbm{1}_{m}\in \mathbb{R}^m$ is a vector all whose coordinates are $1$.
%The following statements are equivalent.
%\begin{enumerate}
%\item $A$ is an Erd\H{o}s matrix. 
%\item ${\bf x}=(x_1, \ldots, x_m)^{T}
%\in \R^m$ solves the system of equations $M{\bf x}=\langle M{\bf x}, {\bf x}\rangle \mathbbm{1}_m$ where $M\in \mathbb{R}^{m\times m}$ is a matrix such that $M_{i, j}= \langle P_i, P_j\rangle_{\operatorname{F}}$ and $\mathbbm{1}_{m}\in \mathbb{R}^m$ is a vector all whose coordinates are $1$. 
%\end{enumerate}
\end{proposition}
In this proposition (and hereafter) we use the notation $\langle\cdot, \cdot\rangle$ to denote the usual inner product in $\R^m$. This proposition provides an algorithmic procedure to test and generate all Erd\H{o}s matrices. This can be used to compute all the Erd\H{o}s matrices--at least in small dimensions. Following the proof of Theorem~\ref{thm:MainTheorem} and the above proposition one can obtain a somewhat better algorithm to generate all the Erd\H{o}s matrices as we explain later. To illustrate this, we revisit the case of dimension $n=3$ in Section~\ref{sec:Dim3} where we compute all $3\times 3$ Erd\H{o}s matrices (up to the equivalence). While the results in Section~\ref{sec:Dim3} are not new, we feel that our method makes the results of~\cite{bouthat2024question} more transparent and provides a more conceptual understanding of the Erd\H{o}s matrices. Some of the computations in Section~\ref{sec:Dim3} extend to higher dimensions as well, thus producing a new class of examples of Erd\H{o}s matrices in all dimensions. This yields a lower bound of $p(n)$, the number of partitions of $n$, on the number of Erd\H{o}s matrices in $\Bcal_n$ (see Proposition~\ref{prop:LowerBound}). %We believe that our proof method may play a useful role in the complete resolution of Question~\ref{ques:Erdos}. 

\sloppy We now state our first main result. The proofs of Proposition~\ref{prop:Characterization} and Theorem~\ref{thm:MainTheorem} are deferred to Section~\ref{sec:Proof}.
\begin{theorem}
\label{thm:MainTheorem}
Let $n\geq 4$. There are only finitely many Erd\H{o}s matrices in $\Bcal_n$.  More precisely, 
\[\Big|\{A\in \Bcal_n: \|A\|_{\operatorname{F}}^2=\maxtr{A}\}\Big|\leq \sum_{j=1}^{(n-1)^2+1}\binom{n!}{j}\;.\]
\end{theorem}

\begin{remark}
\remlab{UpperBoundRemark}
Following our arguments in Section~\ref{sec:Dim3}, we can obtain 
\[\Big|\{A\in \Bcal_n: \|A\|_{\operatorname{F}}^2=\maxtr{A}\}/\sim\Big|\leq \sum_{j=0}^{(n-1)^2}\binom{n!-1}{j}\;.\]
The upper bounds in Theorem~\ref{thm:MainTheorem} as well as in this remark are far from sharp and can be improved. It would be interesting to determine an asymptotic for the quantity $\big|\{A\in \Bcal_n: \|A\|_{\operatorname{F}}^2=\maxtr{A}\}\big|$. This requires some delicate combinatorics that we do not pursue in this paper. 
\end{remark}
 
Using the insights from the proof of Theorem~\ref{thm:MainTheorem} and the computations in Section~\ref{sec:Dim3}, we prove the refinement of Proposition~\ref{prop:Characterization}.
\begin{proposition}%[Corollary~\ref{cor:CompleteCharacterization}]
\label{prop:CompleteChar}
    Let $\{P_1, \ldots, P_m\}\subseteq \Pcal_n$ be a linearly independent collection of permutation matrices. Let $M\in R^{m\times m}$ be the positive definite matrix such that $M_{i, j}=\langle P_i, P_j\rangle_{\operatorname{F}}$. Set
    \[ {\bf x} = \frac{M^{-1}\mathbbm{1}_m}{\langle \mathbbm{1}_m, M^{-1}\mathbbm{1}_m\rangle}\;.\]
    If $x_i\geq 0$ for all $i\in [m]$, then $\sum_{i=1}^{m}x_iP_i$ is a bistochastic matrix and every Erd\H{o}s matrix is of this form. 
\end{proposition}
Finally, we state the second main result of this paper answering~\cite[Question 3]{bouthat2024question}.
\begin{theorem}
\label{thm:RationalEntries}
    Every Erd\H{o}s matrix $A\in \Bcal_n$ has only rational entries. 
\end{theorem}

\subsection*{Outline of the paper}
In Section~\ref{sec:AlphaErdos}, we introduce a natural generalization of Erd\H{o}s matrices and propose some questions. This section can be skipped without affecting the understanding of later sections. 
We prove Proposition~\ref{prop:Characterization} and Theorem~\ref{thm:MainTheorem} in Section~\ref{sec:Proof}. In Section~\ref{sec:Dim3}, we rederive all $3\times 3$ Erd\H{o}s matrices using the insights from our proof technique. Building on the computations in Section~\ref{sec:Dim3}, we complete the proof of Proposition~\ref{prop:CompleteChar} and Theorem~\ref{thm:RationalEntries} in Section~\ref{sec:refinement}.

\section{\texorpdfstring{$\alpha$-Erd\H{o}s matrices}{alpha-Erd\H{o}s matrices}}
\label{sec:AlphaErdos}
In this section, we propose a natural generalization of the Erd\H{o}s matrices (namely $\alpha$-Erd\H{o}s matrices for a suitable range of $\alpha$) and propose several problems that may be of independent interest as well as may be useful in understanding the Erd\H{o}s matrices. 
%A reader can safely skip this section without affecting the understanding of the later sections.

To describe the generalization, we begin with some notations and terminology.  Recall that $\Bcal_n$ denotes the set of all $n\times n$ bistochastic matrices.  We define the function $\Delta_n:\Bcal_n\to \R$ by 
\[ \Delta_n(A) = \max_{\sigma\in S_n} A_{i, \sigma(i)}-\|A\|_{\operatorname{F}}^2\;.\]
Observe that $\Delta_n$ is invariant under the pre- and post-multiplication by a permutation matrix, that is, $\Delta_n(A)=\Delta_n(P_1AP_2)$ where $P_1$ and $P_2$ are $n\times n$ permutation matrices. In other words, $\Delta_n(A)=\Delta_n(B)$ if $A\sim B$. The Marcus--Ree inequality~\Eqref{MarcusRee} is equivalent to $\Delta_n\geq 0$ on $\Bcal_n$.  And, characterizing Erd\H{o}s matrices (up to equivalence) amounts to characterizing the zero-set of $\Delta_n$, that is, $\{A\in \Bcal_n: \Delta_n(A)=0\}$ (or $\{A\in \Bcal_n: \Delta_n(A)=0\}/\sim$). In a forthcoming work, Ottolini and Tripathi~\cite{ottolini2024} study the properties of $\Delta_n(A)$ where $A$ is a random bistochastic matrix drawn according to some probability measure on the set $\Bcal_n$. The following observation was made in~\cite{ottolini2024}.

\begin{proposition}
\label{prop:MaxDiscrepency}
For each $n\geq 1$, we have    \[\max_{A\in \mathcal{B}_n} \Delta_n(A)=(n-1)/4\;.\]
Moreover, the maximum of $\Delta_n$ on $\Bcal_n$ is achieved by the unique (up to the equivalence) matrix  $M_n\coloneqq \frac{1}{2}I_n+\frac{1}{2}J_n$. 
\end{proposition}
\begin{proof}
%By Marcus--Ree inequality~\Eqref{MarcusRee} we know that $\Delta_n(A)\geq 0$. It remains to show that $\Delta_n(A)\leq (n-1)/4$ on $\Bcal_n$.  
%Let $A\in \Bcal_n$, it suffices to show that $\Delta_n(A)\leq (n-1)/4$. To this end, 
Let $A\in \Bcal_n$ and $\sigma$ be the permutation that achieves the maximum in $\Delta_n(A)$. Then, $\Delta_n(A) = \sum_{i=1}^{n} R_i$ where
\begin{align*}
    R_i=\left(A_{i, \sigma(i)} - \sum_{j=1}^{n}A_{i, j}^2\right)\;.
\end{align*}
It suffices to show that $R_i\leq \frac{n-1}{4n}$ for each $1\leq i\leq n$.  To this end, note that (after some relabelling) 
\[R_i = x_1 -x_1^2 -\sum_{i=2}^{n}x_i^2 \leq x_1-x_2^2-\frac{(1-x_1)^2}{n-1},\]
for some $x_i\geq 0$ such that $\sum_{i=2}^{n}x_i=1-x_1$. The last inequality uses  $\sum_{i=^2}^{n}x_i^2\geq (1-x)^2/(n-1)$ which follows from AM-GM inequality.  It follows that
\[ R_i\leq \max_{x\in [0, 1]}\left(x - x^2-\frac{(1-x)^2}{n-1}\right)\;.\]
A simple calculation shows that $R_i\leq \frac{n-1}{4n}$ and that  the maximum is achieved precisely when $x=\frac{1}{2}+\frac{1}{2n}$ and $x_i=\frac{1}{2n}$ for $2\leq i\leq n$.  Since $\Delta_n(A)=\frac{n-1}{4}$ if and only if $R_i=\frac{n-1}{4n}$ for each $i\in [n]$, we conclude that $\Delta_n(A)=\frac{n-1}{4}$ precisely when
\begin{align*}
    A_{i, \sigma(i)} &= \frac{1}{2}+\frac{1}{2n},\quad A_{i, j} = \frac{1}{2n} \quad \forall i, j\in [n], j\neq \sigma(i)\;.
\end{align*}
This completes the proof.
\end{proof}
Since $\Delta_n$ is a continuous function, it follows that $\Delta_n(\Bcal_n)=[0, (n-1)/4]$.  It naturally raises the following generalization of the Erd\H{o}s's question.
\begin{question}
    Let $n\geq 2$ and $\alpha\in [0, (n-1)/4]$.  Characterize the set \[\Bcal_{n, \alpha}\coloneqq \{A\in \Bcal_n: \Delta_n(A)=\alpha\}\] 
 (up to the equivalence). We refer to a matrix $A\in \Omega_{n, \alpha}$ as an $\alpha$-Erd\H{o}s matrix.
\end{question}
As very little is known about this problem, we summarize the progress so far on this problem for the reader's convenience.  The $n=2$ case is trivial. We fully understand the set $\Omega_{3, 0}$ thanks to~\cite{bouthat2024question} and $\Omega_{n, (n-1)/4}$ due to~\cite{ottolini2024}. And, essentially this is all that is known currently. We close this section with the complete solution of $\alpha$-Erd\H{o}s matrix in dimension $n=2$. The following computations are easy exercises, but we include them for completeness.

\subsection{The \texorpdfstring{$n=2$}{n=2} case}
\begin{paragraph}{I.}
\label{par:twobytwo}
A $2\times 2$ bistochastic matrix $A$ looks like $\begin{pmatrix}
p & 1-p \\ 1-p & p
\end{pmatrix}$ for some $p\in [0, 1]$.  In particular, $A$ is an Erd\H{o}s matrix precisely when 
\[ 2(p^2+ (1-p)^2) = \max\{2p, 2(1-p)\}\;.\]
This yields either $p=0, \frac{1}{2}$ or $1$.  In other words, the Marcus--Ree inequality is saturated by the following three matrices:
\[
A_0\coloneqq \begin{pmatrix}
1 & 0 \\ 0 & 1
\end{pmatrix}\,\quad A_{\frac{1}{2}}\coloneqq \frac{1}{2}\begin{pmatrix}
1 & 1 \\ 1 & 1
\end{pmatrix}, \;\quad \text{and}\quad A_{1}\coloneqq \begin{pmatrix}
0 & 1\\ 1 & 0
\end{pmatrix}\;.\]
Since $A_0\sim A_1$, it follows that there are exactly $2$ matrices in $\Bcal_2$ (up to pre- and post-multiplication by permutation matrices) that are Erd\H{o}s matrices.
\end{paragraph}
\begin{paragraph}{II.}
More generally, in this case, one can characterize $\Omega_{2, \alpha}$ for any $\alpha\in [0, 1/4]$.  For completeness, we record it here. Notice that $A=\begin{pmatrix}
    p & 1-p \\ 1-p & p
\end{pmatrix}\in \Omega_{2, \alpha}$ precisely when $2(p^2+ (1-p)^2) - 2\max\{p, (1-p)\}+\alpha=0$. It is easily verified that the solution is given by 
\[p\in \left\{\frac{1\pm \sqrt{1-4\alpha}}{4}\;,\frac{3\pm \sqrt{1-4\alpha}}{4}\right\}\;.\] 
In other words, $|\Omega_{2, \alpha}/\sim|=2$ for all $\alpha\in [0, 1/4)$ and $|\Omega_{2, 1/4}/\sim|=1$.
\end{paragraph}
\begin{remark}
\remlab{alphaErdos}
Note that $2\times 2$ bistochastic matrices have only one free variable and for given $\alpha\in [0, 1/4]$, the problem of finding an $\alpha$-Erd\H{o}s matrix reduces to solving a quadratic equation which can have at most two solutions. 
A $3\times 3$ bistochastic matrix is determined by $4$ free variables. The condition that $\|A\|_{\operatorname{F}}^2=\maxtr{A}-\alpha$, in general, will determine a region in $3$-dimension space. It will be interesting, therefore, to determine if there are only finitely many $\alpha$-Erd\H{o}s matrices in dimension $n=3$. An approach similar to~\cite{bouthat2024question} may be useful here.
\end{remark}

%%%%%%%%%%%%%%%%%%%%%%%%%%%%%%%%%%%%%%%
\section{Proofs}
\label{sec:Proof}
In this section, we prove Proposition~\ref{prop:Characterization} and Theorem~\ref{thm:MainTheorem}. 
\begin{proofof}{Proposition~\ref{prop:Characterization}:} 
Let $A\in \Bcal_n$ be given by 
\[A = \sum_{i=1}^{m}x_iP_i,\]
where $\sum_{i=1}^{m}x_i=1$ and $x_i>0$ for all $i\in [m]$. Observe that %$A$ is an Erd\H{o}s matrix if and only if
\begin{equation}
\eqlab{NormExpansion}
     \|A\|_{\operatorname{F}}^2 = \sum_{i=1}^{m}x_i\langle A, P_i\rangle_{\operatorname{F}}=\sum_{i, j=1}^{m}x_ix_j\langle P_i, P_j\rangle_{\operatorname{F}}.
\end{equation}
Since $\langle A, P_i\rangle_{\operatorname{F}}\leq\maxtr{A}$ for each $i\in [m]$, it follows that $A$ is an Erd\H{o}s matrix if and only if 
\begin{equation}
\eqlab{KeyEquality}
    \|A\|_{\operatorname{F}}^2=\langle A, P_i\rangle_{\operatorname{F}} = \maxtr{A},
\end{equation}
for all $i\in [m]$.
Let $M$ be the symmetric $m\times m$ matrix such that $M_{i, j}= \langle P_i, P_j\rangle_{\operatorname{F}}$ and let ${\bf x} = (x_1, \ldots, x_m)^{T}\in \R^m$. Observe that $\|A\|_{\operatorname{F}}^2 = \langle M{\bf x}, {\bf x}\rangle$. Combining~\Eqref{NormExpansion} and~\Eqref{KeyEquality}, we obtain that if $A$ is an Erd\H{o}s matrix then 
\begin{equation}
\eqlab{FinalEquation}
M{\bf x}= \langle M{\bf x}, {\bf x}\rangle\mathbbm{1}_m\;, 
\end{equation}
where $\mathbbm{1}_m\in \R^m$ is the vector all whose entries are $1$. 
\end{proofof}
The proof of Theorem~\ref{thm:MainTheorem} builds on Proposition~\ref{prop:Characterization} and some lemmas that we prove below. The main idea in the proof of Theorem~\ref{thm:MainTheorem} is that~\Eqref{FinalEquation} has at most one solution if the collection $\{P_1, \ldots, P_m\}$ is \emph{affinely independent} (See Definition~\ref{def:AffInd}). Of course, we also need to prove that every $A\in \Bcal_n$ can be written as a convex combination of an affinely independent collection $\{P_1, \ldots, P_m\}$ of permutation matrices. This is done in~\lemref{redAffine}. \newline

\noindent We now state the following definition (See~\cites[Section 3.5.3]{leonard2015geometry}).
\begin{definition}[Affine independence]
\label{def:AffInd}
A finite collection of vectors $\{x_1, \ldots, x_m\}$ in some Hilbert space $\Hcal$ is said to be \emph{affinely dependent}, if there exists real numbers $c_1, \ldots, c_m$, not all zeros, such that  \[\sum_{i=0}^{m}c_ix_i=0\;,\]
and $\sum_{i=1}^{m}c_i=0$. The set of vectors $\{x_1, \ldots, x_m\}$ is said to be \emph{affinely independent} if it is not affinely dependent. 
\end{definition}

\begin{lemma}
\lemlab{redAffine}
Let $X\in \R^{d}$ be a non-empty convex subset and let $\Ecal$ be the set of extreme points of $X$. Then,  Every $x\in X$ can be written as a convex combination $x=\sum_{i=1}^{m}c_ie_i$ such that $\{e_1, \ldots, e_m\}\subseteq \Ecal$ is a collection of affinely independent vectors and $c_i>0$ for all $i\in [m]$. 
\end{lemma}
\begin{proof}
	The proof follows by fairly standard arguments but we include it for completeness.  A well-known result due to Carath\'eodory in convex geometry~\cites[Theorem 3.3.10]{leonard2015geometry} states that every $x\in X$ can be written as a convex combination of at most $d+1$ extremal points. That is, $x=\sum_{i=1}^{d+1}c_ie_i$ where $\{e_1, \ldots, e_{d+1}\}\subseteq \Ecal$ and $c_i\geq 0$ for all $i\in [m]$ and $\sum_{i=1}^{d+1}c_i=1$.  

    Fix $x\in X$ and let $x=\sum_{i=1}^{m}c_ie_i$ be a minimal representation of $x$ with respect to the number of non-zero coefficients $c_i$. We claim that $\{e_1, \ldots, e_m\}$ is affinely independent. 
	If not, then there exists $0\neq \beta= (\beta_1, \ldots, \beta_m)^{t}\in \R^m$ such that $\sum_{i=1}^{m}\beta_i=0$ and $\sum_{i=1}^{m}\beta_ie_i=0$.  Let 
	\[\alpha\coloneqq \max\{t: t\;|\beta_i|\leq c_i\;\; \forall i\in [m]\}\;.\]
	Notice that there exists some $i_0\in [m]$ such that $c_{i_0}+\alpha\beta_{i_0}=0$ and $(c_i+\alpha\beta_i)\geq 0$ for all $i\in [m]$ by construction. Further observe that $x= \sum_{i=1}^{m}(c_i+\alpha\beta_i)e_i$, but this contradicts the fact that $x=\sum_{i=1}^{m}c_ie_i$ was a minimal representation (with respect to the number of non-zero coefficients). 
\end{proof}
It follows from~\lemref{redAffine} that every bistochastic matrix $A\in \Bcal_n$ can be written as a convex combination of a collection of permutation matrices $\{P_1, \ldots, P_m\}$ that is affinely independent. We now show that if ${P_1, \ldots, P_m}$ is an affinely independent collection of permutation matrices, then~\Eqref{FinalEquation} has at most one solution. More generally, we make the following observation.

\begin{lemma}
\label{lem:Uniqueness}
Let $\{f_1, \ldots, f_m\}$ be a collection of affinely independent vectors in some Hilbert space $\mathcal{H}$ with the inner-product $\langle \cdot, \cdot\rangle_{\Hcal}$.  Let $M$ be the $m\times m$ matrix such that $M_{i,j}= \langle f_i, f_j\rangle_{\Hcal
}$.  Let $u, v\in \R^{m}$ be such that $Mu= Mv$. If $\sum_{i=1}^{m}u_i = \sum_{i=1}^{m}v_i$ then $u=v$. 
\end{lemma}
\begin{proof}
	Set $w= u-v$. Since $Mw=0$, we obtain that 
	\[\sum_{j=1}^{m}\langle f_i, f_j\rangle_{\Hcal} w_j = 0, \quad \forall i=1, \ldots, m\;.\]
	For any vector $\widetilde{w}\in \R^m$, we obtain 
	\begin{align*}
	0=\sum_{i=1}^{m} \widetilde{w}_i\sum_{j=1}^{m}\big\langle f_i, f_j\big\rangle_{\Hcal} w_j = \Big\langle \sum_{i=1}^{m} \widetilde{w}_if_i, \sum_{j=1}^{m} w_jf_j\Big\rangle_{\Hcal} &=0\;.
	\end{align*}
	Taking $\widetilde{w}=w$, we conclude that $\sum_{i=1}^{m}w_if_i=0$. Since $f_1, \ldots, f_m$ are affinely independent and $\sum_{i=1}^{m}w_i=0$, it follows that $w_i=0$ for all $i=1, \ldots, m$.
\end{proof}
The proof of Theorem~\ref{thm:MainTheorem} follows easily from Proposition~\ref{prop:Characterization} and Lemma~\ref{lem:Uniqueness} but we include it for completeness. 

\subsection{Proof of Theorem~\ref{thm:MainTheorem}}
For an affinely independent collection of permutation matrices $\{P_1, \ldots, P_m\}$, we define
 \[\operatorname{Co}\left(\{P_1, \ldots, P_m\}\right)\coloneqq \left\{\sum_{i=1}^{m}\alpha_iP_i: \alpha_i>0, \;\;\sum_{i=1}^{m}\alpha_i=1\right\}\;.\]
 It follows from~\lemref{redAffine} that every $n\times n$ bistochastic matrix $A\in \operatorname{Co}(\{P_1, \ldots, P_m\})$ for some collection of affinely independent permutation matrices $\{P_1, \ldots, P_m\}$. Furthermore, Lemma~\ref{lem:Uniqueness} shows that for any such collection of affinely independent permutation matrices $\{P_1, \ldots, P_m\}$, there is at most one Erd\H{o}s matrix in $\operatorname{Co}(\{P_1, \ldots, P_m\})$. Since there are only finitely many permutation matrices, we conclude there are only finitely many Erd\H{o}s matrices. 

As $\Bcal_n$ is a convex subset of dimension $(n-1)^2$, it follows Carath\'eodory's theorem~\cite[Theorem 3.3.10]{leonard2015geometry} that a bistochastic matrix in $\Bcal_n$ can be written as a convex combination of at most $(n-1)^2+1$ permutation matrices. Since there are at most $\binom{n!}{m}$ many affinely independent collection of permutation matrices of size $m$, it follows that
 \[\big|\{A\in \Bcal_n: \|A\|_{\operatorname{F}}^2=\maxtr{A}\}\big|\leq \sum_{j=1}^{(n-1)^2+1}\binom{n!}{j}\;.\]

%\subsection{Proof of Proposition~\ref{prop:MaxDiscrepency}}
%\label{subsec:PropositionProof}

\section{An algorithm for Erd\H{o}s matrices}
\label{sec:algorithm}
Recall that, up to the equivalence, there are only $6$ Erd\H{o}s matrices in $\Bcal_3$ and these are given by $I_3, J_3, I\oplus J_2, S, R$, and $T$.  This was shown in~\cite{bouthat2024question}. Their proof can be broken into three ingredients.  The first ingredient is a result, due to Marcus and Ree, that states that except $J_3$ any Erd\H{o}s matrix must have a zero entry.  Using the equivalence, it suffices to consider the matrices of the form 
\[A\coloneqq \begin{pmatrix}
x & w &  1-x-w \\
0 & y & 1-y \\
1-x & 1-w-y &  x+2w+2y-2
\end{pmatrix}\;.\]

The authors in~\cite{bouthat2024question} characterize the matrices of the above form that satisfy $\|A\|_F^2 = \mathrm{Trace}(A)$.  Since $w$ has to be a real number, this condition determines a feasible region for $(x, y)$ in $\R^2$, and for $(x, y)$ in this region the parameter $w= w(x, y)$ has at most two possible values that are obtained by solving a quadratic equation in $w$.  This is essentially the second ingredient of the proof. The third and final ingredient is as follows. For each of the two choices of the $w\equiv w(x, y)$ over the feasible region, the authors find the part of the feasible region such that $\mathrm{Trace}(A)=\maxtr{A}$. And, curiously, this yields only finitely many values for $(x, y)$. 

In the following, we re-derive the same result using insights from Proposition~\ref{prop:Characterization} and the proof of Theorem~\ref{thm:MainTheorem}. The key idea is very simple. Our proof suggests the following general recipe for generating Erd\H{o}s matrices. 

\begin{enumerate}
\item Enumerate the collection $\mathcal{C}$ of \emph{affinely independent} sets $\{P_1, \ldots, P_m\}$ for $1\leq m\leq (n-1)^2+1$.  
\item Given an affinely independent collection $\{P_1, \ldots, P_m\}\in \mathcal{C}$, construct the $m\times m$ matrix $M$ such that $M_{i, j} = \langle P_i, P_j\rangle_{\operatorname{F}}$. Notice that the matrix $M$ is a Gram matrix and therefore it is positive semidefinite~\cite[Theorem 7.2.10]{horn2012matrix}. 
\item Solve for $M{\bf x}=\langle M{\bf x}, {\bf x}\rangle \mathbbm{1}_m$ and $x_i\geq 0$ for all $i$ and $\langle \mathbbm{1}_m, {\bf x}\rangle=1$, if it exists.
\item Check if the matrix $A = \sum_{i=1}^{m}x_iP_i$ is an Erd\H{o}s matrix. 
\end{enumerate}
\begin{remark}
Notice that $M_{ii}=\langle P_i, P_i\rangle_{\operatorname{F}}=n$ and hence for any vector ${\bf x}\neq 0$ such $x_i\geq 0$, we have that $\langle M{\bf x}, {\bf x}\rangle\geq n\sum_{i=1}^{m}x_i^2>0$. On the other hand, if ${\bf x}$ solves $M{\bf x}=\langle M{\bf x}, {\bf x}\rangle \mathbbm{1}_m$ and $x_i\geq 0$, then $\langle M{\bf x}, {\bf x}\rangle =\langle M{\bf x}, {\bf x}\rangle \langle \mathbbm{1}_m, {\bf x}\rangle$. In particular, if ${\bf x}\neq 0$ and $x_i>0$ then $\langle \mathbbm{1}_m, {\bf x}\rangle=1$. Therefore, the last condition in step 3 above is superfluous. 
\end{remark}

\begin{remark}
\remlab{ImprovedUpperBound}
Since we are interested in characterizing the Erd\H{o}s matrices only up to equivalence, one can, without loss of generality, assume that $I_n\in \{P_1, P_2, \ldots, P_m\}$ in the step $(1)$ above. This yields the upper bound in~\remref{UpperBoundRemark}.
\end{remark}
Declare two families of permutation matrices $\{P_1, P_2, \ldots, P_m \}$ and $\{Q_1, Q_2, \ldots, Q_m\}$ to be \emph{equivalent} if there exist permutation matrices $P, Q$ such that \[\{PP_1Q, PP_2Q, \ldots, PP_mQ\} = \{Q_1, Q_2, \ldots, Q_m\}.\] 
Let $\mathcal{O}_{m, n}$ denote the number of non-equivalent subsets $\{P_1, \ldots, P_m\}\subseteq \Pcal_n$. Then, the upper bound in~\remref{UpperBoundRemark} can be improved to $\sum_{m=1}^{(n-1)^2+1}\mathcal{O}_{m, n}$. 
%It would be interesting to understand the growth of $\mathcal{O}_{m, n}$.

\subsection{Dimension \texorpdfstring{$n=3$}{n=3}: Revisited}
\label{sec:Dim3}
We now begin the case of dimension $n=3$. It will be useful to set some notations for this section.  Let $S_3=\{e, \sigma, \gamma, \delta, \rho, \rho^2\}$ denote the symmetric group on the set $\{1, 2, 3\}$ where 
\begin{align*}
\sigma = (12),\quad \gamma=(23), \quad \delta= (13), \quad \rho=(123)\;.
\end{align*}
Throughout this section, we will identify a permutation $\pi\in S_3$ with the corresponding permutation matrix via the left-multiplication and we will denote it by $P_{\pi}$. % defined by $P_{\pi}(i, j)=1$ if $\pi(i)=j$ and $0$ otherwise. 
With this convention, we have that $\mathrm{Trace}(P_{\pi})=\text{No. of fixed points of }\pi$. 
%The following multiplication table of $S_3$ will be useful.\todo{Add the multiplication table.}

By Carath\'eodory's theorem, we know that every $3\times 3$ bistochastic matrix can be written as a convex combination of at most $(3-1)^2+1=5$ permutation matrices. In other words, we only need to consider the cases $1\leq m\leq 5$ in the algorithm described in the previous section. Since we are interested in Erd\H{o}s matrices up to equivalence, we only consider the non-equivalent families of permutations matrices of size $m$.

\subsection*{Case $m=1$}
In this case, the identity matrix is the only candidate (up to equivalence) and it is easily verified to be an Erd\H{o}s matrix.  Of course, this is true in all dimensions.

\subsection*{Case $m=2$}
Observe that the family $\{I, P_{\pi}\}$ yields $M=\begin{pmatrix}3 & d \\ d & 3\end{pmatrix}$ where $d$ is the number of fixed points of the permutation $\pi$.  It is easily checked that ${\bf x}= (1/2, 1/2)^{T}$ solves $M{\bf x}=\langle M{\bf x}, {\bf x}\rangle\mathbbm{1}_2$.  In particular, this yields that $\frac{1}{2}I_3+ \frac{1}{2}P_{\pi}$ is an Erd\H{o}s matrix.

By conjugation, it is easy to see that there are precisely two non-equivalent choices for $\{I, P_{\pi}\}$ corresponding to $\pi=\sigma$ and $\pi=\rho$. This yields the following two (non-equivalent) Erd\H{o}s matrices:
\begin{align*}
\frac{1}{2}(I_3+P_{\sigma}) &= \begin{pmatrix}\frac{1}{2} & \frac{1}{2} & 0 \\ \frac{1}{2} & \frac{1}{2} & 0 \\ 0 & 0 & 1 \end{pmatrix}\sim I\oplus J_2\;,\\
\frac{1}{2}(I_3+P_{\rho}) &= \begin{pmatrix}
\frac{1}{2} & \frac{1}{2} & 0 \\ 0 & \frac{1}{2} & \frac{1}{2} \\ \frac{1}{2} & 0 & \frac{1}{2}
\end{pmatrix}\sim T\;.
\end{align*}
The above computation easily extends to higher dimensions and yields the following proposition.
\begin{proposition}
\label{prop:LowerBound}
Let $n\in \mathbb{N}$ and $P$ be an $n\times n$ permutation matrix. Let $A\coloneqq \frac{1}{2}(I_n+ P)$. Then $A$ is an Erd\H{o}s matrix, that is,
\[\|A\|_F^2 = \frac{n+d}{2}= \mathrm{maxTr}(A)\;.\]
Furthemore, $\{I_n, P_1\}$ and $\{I, P_2\}$ are equivalent if and only if $P_1$ and $P_2$ are conjugates.  
\end{proposition}
It is a well-known fact that the number of conjugacy classes in the symmetric group $S_n$ is equal to the number of integer partitions $p(n)$ of $n$. The above proposition yields a lower bound of $p(n)$ on the number of non-equivalent Erd\H{o}s matrices in $\Bcal_n$.

\subsection*{Case $m=3$}
In this case, we consider the sets of the form $\{I_3, P_{\pi_1}, P_{\pi_2}\}$.  There are $10$ different choices for such sets. However, many of these are equivalent. For completeness, we give the details below. Let us begin with the set $\{I_3, P_{\sigma}, P_{\gamma}\}$.  In this case, the matrix $M$ is given by 
\[M \coloneqq \begin{pmatrix}
3 & 1 & 1 \\ 1 & 3 & 0 \\ 1 & 0 & 3
\end{pmatrix}\;.\]

It is easily verified that ${\bf x} = (1/5, 2/5, 2/5)^{T}$.  This yields the following Erd\H{o}s matrix 
\[\frac{1}{5}I_3 + \frac{2}{5}P_{\sigma}+\frac{2}{5}P_{\gamma} = \begin{pmatrix}
\frac{3}{5} & \frac{2}{5} & 0 \\
\frac{2}{5} & \frac{1}{5} & \frac{2}{5}\\
0 & \frac{2}{5}& \frac{3}{5} 
\end{pmatrix}\sim R \;.\]

Note that the set $\{I_3, P_{\sigma}, P_{\gamma}\}$ is equivalent to $\{I_3, P_{\sigma}, P_{\delta}\}, \{I_3, P_{\gamma}, P_{\delta}\}$ by conjugation. Furthermore, it is also equivalent to the following 
\begin{align*}
&\{I_3, P_{\sigma}, P_{\gamma}P_{\sigma}=P_{\rho^2}\}, \quad \{I_3, P_{\sigma}, P_{\gamma}P_{\sigma}=P_{\rho}\}, \\
&\{I_3, P_{\gamma}, P_{\gamma}P_{\sigma}=P_{\rho}\}, \quad \{I_3, P_{\gamma}, P_{\sigma}P_{\gamma}=P_{\rho^2}\}\;.
\end{align*}
And, similarly, it is also equivalent to the sets $\{I_3,  P_{\delta}, P_{\rho}\}$ and $\{I_3,  P_{\delta}, P_{\rho^2}\}$. This leaves us with the set $\{I_3, P_{\rho}, P_{\rho^2}\}$.  In this case, the matrix $M=3I_3$  and ${\bf x}=(1/3, 1/3, 1/3)^T$. This yields the following Erd\H{o}s matrix
\[
\frac{1}{3}(I_3+ P_{\rho}+ P_{\rho^2}) = J_3\;.
\]

\subsection*{Case $m=4$}
We again have $10$ different choices for the family $\{I_3, P_{\pi_1}, P_{\pi_2}, P_{\pi_3}\}$. We begin with the set $\{I_3, P_{\sigma}, P_{\gamma}, P_{\delta}\}$.  We skip the simple verification that it is equivalent to the following sets 
\begin{align*}
\{I_3, P_{\sigma}, P_{\rho}, P_{\rho^2}\}, \{I_3, P_{\gamma}, P_{\rho}, P_{\rho^2}\}, \{I_3, P_{\delta}, P_{\rho}, P_{\rho^2}\}\;.
\end{align*}
In this case, the matrix $M$ is given by 
\[
\begin{pmatrix}
3 & 1 & 1 & 1 \\
1 & 3 & 0 & 0 \\
1 & 0 & 3 & 0 \\
1 & 0 & 0 & 3
\end{pmatrix}\;,
\]
and it is easy to check that ${\bf x} = (0, 1/3, 1/3, 1/3)$ satisfies $M{\bf x}=\langle M{\bf x}, {\bf x}\rangle$. This yields the Erd\H{o}s matrix $0\cdot I_3+ \frac{1}{3}(P_{\sigma}+P_{\gamma}+P_{\delta})=J_3$. Finally, consider the set $\{I_3, P_{\sigma}, P_{\gamma}, P_{\rho}\}$ (which is equivalent to the remaining possibilities) and verify that 
\[M =\begin{pmatrix}
3 & 1 & 1 & 0 \\
1 & 3 & 0 & 1 \\
1 & 0 & 3 & 1 \\
0 & 1 & 1 & 3
\end{pmatrix}\;,\]
and ${\bf x}= (1/4, 1/4, 1/4, 1/4)^{T}$ satisfies $M{\bf x}=\langle M{\bf x}, {\bf x}\rangle$. This yields the following Erd\H{o}s matrix
\[\frac{1}{4}(I_3+P_{\sigma}+P_{\gamma}+P_{\rho})=\begin{pmatrix}
\frac{1}{2} & \frac{1}{2} & 0 \\
\frac{1}{4} & \frac{1}{4} & \frac{1}{2}\\
\frac{1}{4} & \frac{1}{4} & \frac{1}{2} 
\end{pmatrix}\sim S\;.\]

\subsection*{Case $m=5$}
It is easily verified that any collection of $5$ permutation matrices is equivalent to $\{I_3, P_{\rho}, P_{\sigma}, P_{\gamma}, P_{\delta}\}$. This yields
\[M=\begin{pmatrix}
3 & 0 & 1 & 1 & 1 \\
0 & 3 & 1 & 1  & 1 \\
1 & 1 & 3 & 0 & 0 \\
1 & 1 & 0 & 3 & 0 \\
1 & 1 & 0 & 0 & 3 
\end{pmatrix}\;,\]
and ${\bf x}=(0, 0, 1/3, 1/3, 1/3)^{T}$ and hence the Erd\H{o}s matrix $J_3$.

\begin{remark}
In dimension $n=3$, any collection of permutation matrices $\{P_1, \ldots, P_m\}$ for $m\leq 5$ is, in fact, linearly independent (see~\cite{louck2011applications}). Therefore, we do not need to check the affine independence condition.  
\end{remark}

\subsection{Further refinement}
\label{sec:refinement}
In this section, we show that it is enough to consider the collection of permutation matrices $\{P_1, \ldots, P_m\}$ that are linearly independent (not affinely independent) in the algorithm described in Section~\ref{sec:algorithm}.

\begin{proposition}
    Let $\{P_1, \ldots, P_m\}$ be an affinely independent collection of permutation matrices that is not linearly independent. Let $M\in \R^{m\times m}$ be the matrix such that $M_{i, j}=\langle P_i, P_j\rangle_{\operatorname{F}}$. Let ${\bf x}$ be the unique solution to 
    \begin{equation}
     \eqlab{MasterEquation}
        M{\bf x} = \langle M{\bf x}, {\bf x}\rangle\mathbbm{1}_m, \quad \langle \mathbbm{1}_m, {\bf x}\rangle =1\;.
    \end{equation}
Then, ${\bf x}$ has at least $1$ zero entry. 
\end{proposition}
\begin{proof}
    Since $\{P_1, \ldots, P_m\}$ is affinely independent, (possibly after some relabelling) we can assume that there exist real numbers $\beta_{i}, i\in [m-1]$ such that $\sum_{i=1}^{m-1}\beta_i=1$ and 
    \[ P_{m}= \sum_{i=1}^{m-1}\beta_iP_i\;.\]
    We can assume without loss of generality that $P_1, \ldots, P_{m-1}$ are linearly independent. Let $\widetilde{M}\in \R^{(m-1)\times (m-1)}$ be the $(m-1)\times (m-1)$ principal submatrix of $M$. Notice that $\widetilde{M}=\left\langle P_i, P_j\rangle_{\operatorname{F}}\right)_{1\leq i, j\leq m-1}$ is a Gram matrix. Since $P_1, \ldots, P_{m-1}$ are linearly independent, it follows from~\cite[Theorem 7.2.10]{horn2012matrix} that $\widetilde{M}$ is invertible. Let ${\bf y}$ be the unique solution to 
    \[\widetilde{M}{\bf y}= \langle \widetilde{M}{\bf y}, {\bf y}\rangle \mathbbm{1}_{m-1}, \quad \langle \mathbbm{1}_{m-1}, {\bf y}\rangle =1\;.\]
By Lemma~\ref{lem:Uniqueness}, we conclude that $\begin{pmatrix}
        {\bf y}\\ 0
    \end{pmatrix}$ is the unique solution to~\Eqref{MasterEquation}.   
\end{proof}
This immediately yields the following corollary that sheds light on the structure of Erd\H{o}s matrices.
\begin{corollary}
\label{cor:ErdosStructure}
    Every Erd\H{o}s matrix $A\in \Bcal_n$ can be written as a convex combination of linearly independent permutation matrices. 
\end{corollary}
This simple corollary has important consequences. First of all, it tells us that Step (1) of the above algorithm suffices to consider the linearly independent collection of permutation matrices. Furthermore, if $\{P_1, \ldots, P_m\}$ is a linearly independent collection of permutation matrices, then the matrix $M$ such that $M_{i, j}=\langle P_i, P_j\rangle_F$ is positive definite and hence invertible~\cite[Theorem 7.2.10]{horn2012matrix}. Set ${\bf y}= M^{-1}\mathbbm{1}_m$ and observe that $\langle \mathbbm{1}_m, {\bf y}\rangle= \langle M{\bf y}, {\bf y}\rangle >0$ because ${\bf y}$ is non-zero and $M$ is positive definite. Define ${\bf x} = \frac{{\bf y}}{\langle \mathbbm{1}_m, {\bf y}\rangle}$. By Lemma~\ref{lem:Uniqueness} we obtain that ${\bf x}$ is the unique solution to~\Eqref{MasterEquation}. The proof of Proposition~\ref{prop:CompleteChar} is now immediate. The above discussion also yields the proof of Theorem~\ref{thm:RationalEntries} that we include below.
\iffalse
This leads to the following refinement of Corollary~\ref{cor:ErdosStructure} that in some sense characterizes the Erd\H{o}s matrices. 
\begin{corollary}
\label{cor:CompleteCharacterization}
    Let $\{P_1, \ldots, P_m\}\subseteq \Pcal_n$ be a linearly independent collection of permutation matrices. Let $M\in R^{m\times m}$ the positive definite matrix such that $M_{i, j}=\langle P_i, P_j\rangle_{\operatorname{F}}$. Set
    \[ {\bf x} = \frac{M^{-1}\mathbbm{1}_m}{\langle \mathbbm{1}_m, M^{-1}\mathbbm{1}_m\rangle}\;.\]
    Let $A=\sum_{i=1}^{m}x_iP_i$
    If $x_i\geq 0$ for all $i\in [m]$, then $\sum_{i=1}^{m}x_iP_i$ is a bistochastic matrix and every Erd\H{o}s matrix can 
\end{corollary}
\fi
%Finally, we complete the proof of Theorem~\ref{thm:RationalEntries} that follows immediately from the above discussion.
\begin{proofof}[Theorem~\ref{thm:RationalEntries}]
    Let $A\in \Bcal_n$ be an Erd\H{o}s matrix. By Corollary~\ref{cor:ErdosStructure}, there exists a linearly independent collection of permutation matrices $\{P_1, \ldots, P_m\}$ and ${\bf x}=(x_1, \ldots, x_m)^T$ such that 
    \[ A= \sum_{i=1}^{m}x_iP_i, \quad \sum_{i=1}^{m}x_i = 1, \quad x_i>0\;\;\forall i\in [m]\;.\]
    Let $M$ be defined as $M_{ij} = \langle P_i, P_j\rangle_{\operatorname{F}}$ for $i, j\in [m]$. Let ${\bf x}=\frac{M^{-1}\mathbbm{1}_m}{\big\langle \mathbbm{1}_m, M^{-1}\mathbbm{1}_m\big\rangle}$ be the unique solution to~\Eqref{MasterEquation}.
    %it follows from the above discussion that 
    %\[ {\bf x} = \frac{M^{-1}\mathbbm{1}_m}{\langle \mathbbm{1}_m, M^{-1}\mathbbm{1}_m\rangle}\;.\]
    Since the entries of $M$ are non-negative integers, it follows that $M^{-1}$ has only rational entries, and hence ${\bf x}$ has rational entries. This completes the proof.
\end{proofof}
Notice that in our computations for dimension $n=3$ in Section~\ref{sec:Dim3} we have that ${\bf x}= M^{-1}\mathbbm{1}_m$ always satisfies $x_i\geq 0$. The following example shows in general $M^{-1}\mathbbm{1}_m$ can have negative entries.  
\begin{example}
\label{ex:NonnegativityRowSum}
Let $S_4$ be the group of all permutations of the set $\{1, \ldots, 4\}$. Let $\pi_1$ be the identity permutation and let $\pi_2=(12), \pi_3=(23)$ and $\pi_4=(34)$. Let $P_{i}$ denote the permutation matrix corresponding to the permutation $\pi_i$ (identified via left multiplication) and let $M$ denote the $4\times 4$ matrix such that $M(i, j)=\langle P_i, P_j\rangle_{\operatorname{F}}$. It is easy to verify that 
\[M = \begin{pmatrix}
    4 & 2 & 2& 2\\
    2 & 4 & 1 & 0\\
    2 & 1 & 4 & 1\\
    2 & 0 & 1 & 4
\end{pmatrix}\;, \qquad M^{-1}=\frac{1}{24}\begin{pmatrix} 14 & -6 & -4 & -6\\
-6 & 9 & 0 & 3\\
-4 & 0 & 8 & 0\\
-6 & 3 & 0 & 9
\end{pmatrix}\;. \]
Note that the first entry of $M^{-1}\mathbbm{1}_4$ is $-1/12$ which is negative.  
%This yields that ${\bf x}=\frac{1}{7}(-1, 3, 2, 3)$. 
\end{example}
This example naturally raises the following question which will be needed to obtain the number of non-equivalent Erd\H{o}s matrices in any dimensions. 
\begin{question}
\label{Question:Positivity}
Let $\{P_1, \ldots, P_m\}$ be a linearly independent collection of permutation matrices and let $M$ be the $m
\times m$ matrix such that $M(i, j)=\langle P_i, P_j\rangle_{\operatorname{F}}$. Under what conditions are all the entries of $M^{-1}\mathbbm{1}$ positive?
\end{question}

\section*{Acknowledgments}
I thank Junaid Hasan for suggesting the reference~\cite{louck2011applications}. Example~\ref{ex:NonnegativityRowSum} is adapted from an example I learned from Aman Kushwaha. I also thank Andrea Ottolini and Mayuresh Londhe for several insightful discussions. I am also grateful to Prof. Fr\'ed\'eric Morneau-Gu\'erin for reading the first draft of the paper and for his encouraging remarks. Finally, I must thank two referees for making several important suggestions that greatly improved the paper. In particular, the Example~\ref{ex:NonnegativityRowSum} was found after a question that one of the referees asked.

%-------------------------------------------------------------------------

\BibTexMode{%
   \bibliographystyle{alpha}
   \bibliography{reference}

@article{marcus1959diagonals,
  title={Diagonals of doubly stochastic matrices},
  author={Marcus, Marvin and Ree, Rimhak},
  journal={The Quarterly Journal of Mathematics},
  volume={10},
  number={1},
  pages={296--302},
  year={1959},
  publisher={Oxford University Press}
}

@article{bouthat2024question,
  title={On a question of {E}rd{\H{o}}s on doubly stochastic matrices},
  author={Bouthat, Ludovick and Mashreghi, Javad and Morneau-Gu{\'e}rin, Fr{\'e}d{\'e}ric},
  journal={Linear and Multilinear Algebra},
  pages={1--22},
  year={2024},
  publisher={Taylor \& Francis}
}

@article{ottolini2024,
  title={Maxtrace of random bistochastic matrices},
  author={Ottolini, Andrea and Tripathi, Raghavendra},
  journal={Private communication}
}

@article{birkhoff1946three,
  title={Tres observaciones sobre el algebra lineal},
  author={Birkhoff, Garrett},
  journal={Univ. Nac. Tacuman, Rev. Ser. A},
  volume={5},
  pages={147--151},
  year={1946}
}

@article{balasubramanian1979maximal,
  title={Maximal diagonal sums},
  author={Balasubramanian, K},
  journal={Linear and Multilinear Algebra},
  volume={7},
  number={3},
  pages={249--251},
  year={1979},
  publisher={Taylor \& Francis}
}

@book{horn2012matrix,
  title={Matrix analysis},
  author={Horn, Roger A and Johnson, Charles R},
  year={2012},
  publisher={Cambridge university press}
}

@article{wang1974maximum,
  title={Maximum and minimum diagonal sums of doubly stochastic matrices},
  author={Wang, Edward Tzu-Hsia},
  journal={Linear Algebra and its Applications},
  volume={8},
  number={6},
  pages={483--505},
  year={1974},
  publisher={Elsevier}
}

@article{brualdi2022diagonal,
  title={Diagonal sums of doubly stochastic matrices},
  author={Brualdi, Richard A and Dahl, Geir},
  journal={Linear and Multilinear Algebra},
  volume={70},
  number={20},
  pages={4946--4972},
  year={2022},
  publisher={Taylor \& Francis}
}

@book{leonard2015geometry,
  title={Geometry of convex sets},
  author={Leonard, Isaac E and Lewis, James Edward},
  year={2015},
  publisher={John Wiley \& Sons}
}

@book{louck2011applications,
  title={Applications of unitary symmetry and combinatorics},
  author={Louck, James D},
  year={2011},
  publisher={World Scientific}
}
}%
\BibLatexMode{\printbibliography}

\end{document}